\documentclass[11pt]{amsart}
\usepackage{mathrsfs}
\usepackage{amsfonts}
\usepackage{latexsym,amsmath,amssymb}

 \renewcommand{\epsilon}{\varepsilon}

\newtheorem{theorem}{Theorem}[section]
 
 \newtheorem{lemma}[theorem]{Lemma}
 
 \newtheorem{Corollary}[theorem]{Corollary}
 \newtheorem{proposition}[theorem]{proposition}
 \newtheorem{Proposition}[theorem]{Proposition}
\newtheorem{deff}[theorem]{Definition}
 \newtheorem{rem}[theorem]{Remark}
 \newcommand{\bth}{\begin{theorem}}
 \newcommand{\ble}{\begin{lemma}}
 \newcommand{\bcor}{\begin{corr}}
 \newcommand{\bdeff}{\begin{deff}}
 \newcommand{\bprop}{\begin{proposition}}
 \newcommand{\ele}{\end{lemma}}
 \newcommand{\ecor}{\end{corr}}
 \newcommand{\edeff}{\end{deff}}
 
 \newcommand{\eprop}{\end{proposition}}

 \renewcommand{\Pi}{\varPi}

 \renewcommand{\epsilon}{\varepsilon}

\numberwithin{equation}{section}

\title
[Nonlinear second boundary problem]{On the second boundary value problem for Lagrangian mean curvature equation}

\thanks{The second author is supported by National Natural Science Foundation of China (No. 11771103 and 11871102) and Guangxi Natural Science Foundation (2017GXNSFFA198017). The third author is supported in part by the National Natural Science Foundation of China (11631002 and 11871102).}

\date{}

\begin{document}
\maketitle

\begin{center}Chong Wang \footnote{Fengtai School of the High School Affiliated to Renmin University of China, Beijing 100074, China. wch0229@mail.bnu.edu.cn}
$\cdot$ Rongli Huang \footnote{The corresponding author. School of Mathematics and Statistics, Guangxi Normal University, Guangxi 541004, China. ronglihuangmath@gxnu.edu.cn}
$\cdot$ Jiguang Bao \footnote{School of Mathematical Sciences, Beijing Normal University, Laboratory of Mathematics and Complex Systems, Ministry of Education, Beijing 100875, China. jgbao@bnu.edu.cn}
\end{center}


\begin{abstract}
Considering the second boundary value problem of the Lagrangian mean curvature equation, we obtain the existence and uniqueness of the smooth uniformly convex solution, which generalizes the Brendle-Warren's theorem about minimal Lagrangian diffeomorphism in Euclidean metric space.
\end{abstract}





\section{Introduction}
 The main aim of this article is to study the existence and uniqueness of the smooth uniformly convex solution for the second boundary value problem of the Lagrangian mean curvature equation
\begin{equation}\label{e1.1}
\left\{ \begin{aligned}F_\tau(\lambda(D^2 u))&=\kappa\cdot x+c,\ \ x\in \Omega, \\
Du(\Omega)&=\tilde{\Omega},
\end{aligned} \right.
\end{equation}
where $\Omega$ and $\tilde{\Omega}$ are two uniformly convex bounded domains with smooth boundary in $\mathbb{R}^{n}$, $\kappa\in \mathbb{R}^{n}$ is a constant vector, $\lambda(D^2 u)=(\lambda_1,\cdots, \lambda_n)$ are the eigenvalues of Hessian matrix $D^2 u$, $c$ is a constant to be determined and
\begin{equation}\label{e1.101}
F_{\tau}(\lambda):=\left\{ \begin{aligned}
&\frac{1}{n}\sum_{i=1}^n\ln\lambda_{i}, &&\tau=0, \\
& \frac{\sqrt{a^2+1}}{2b}\sum_{i=1}^n\ln\frac{\lambda_{i}+a-b}{\lambda_{i}+a+b},  &&0<\tau<\frac{\pi}{4},\\
& -\sqrt{2}\sum_{i=1}^n\frac{1}{1+\lambda_{i}}, &&\tau=\frac{\pi}{4},\\
& \frac{\sqrt{a^2+1}}{b}\sum_{i=1}^n\arctan\frac{\lambda_{i}+a-b}{\lambda_{i}+a+b},  \ \ &&\frac{\pi}{4}<\tau<\frac{\pi}{2},\\
& \sum_{i=1}^n\arctan\lambda_{i}, &&\tau=\frac{\pi}{2},
\end{aligned} \right.
\end{equation}
where $a=\cot \tau$, $b=\sqrt{|\cot^2\tau-1|}$.

Let
$$g_\tau=\sin \tau \delta_0+\cos \tau g_0,\ \ \tau\in\left[0,\frac{\pi}{2}\right]$$
be the linear combined metric of the standard Euclidean metric
$$\delta_0=\sum_{i=1}^n dx_i\otimes dx_i+\sum_{j=1}^ndy_j\otimes dy_j$$
and the pseudo-Euclidean metric
$$g_{0}=\frac{1}{2}\sum_{i=1}^n dx_i\otimes dy_i+\frac{1}{2}\sum_{j=1}^n dy_j\otimes dx_j$$
in $\mathbb{R}^{n}\times \mathbb{R}^{n}$.

In 2010, Warren \cite{MW} firstly obtained that if $(x, Du(x))$ is a minimal Lagrangian graph in $(\mathbb{R}^n\times\mathbb{R}^n, g_\tau)$, then $u$ satisfies
\begin{equation}\label{equ0.1.5}
F_\tau(\lambda(D^2 u))=c,
\end{equation}
which is a special case of (\ref{e1.1}) when $\kappa\equiv 0$.

If $\tau=0$, (\ref{equ0.1.5}) is the famous Monge-Amp\`{e}re equation
$$\det D^2u=e^{2c},$$
which general form is
\begin{equation}\label{equ0.1.3}
\det D^2 u=f(x, u, Du).
\end{equation}

If $\tau=\frac{\pi}{2}$, (\ref{equ0.1.5}) becomes the special Lagrangian equation
\begin{equation}\label{equ0.1.1}
\sum_{i=1}^{n}\arctan\lambda_{i}(D^2 u)=c.
\end{equation}
The special Lagrangian equation was first introduced by Harvey and Lawson in \cite{HL}. They proved that a Lagrangian graph $(x, Du(x))$ in $(\mathbb{R}^n\times\mathbb{R}^n, \delta_0)$ is minimal if and only if the Lagrangian angle is a constant, that is, (\ref{equ0.1.1}) holds.
According to (1.5),  several authors  obtained the same  Bernstein type theorems simultaneously using  different techniques. Jost and Xin \cite{JX} used the properties of harmonic maps into convex subsets of Grassmannians. Yuan \cite{Y} used the geometric measure theory.

We will be considering the Lagrangian graphs of  prescribed constant mean curvature $\kappa$ in $(\mathbb{R}^n\times\mathbb{R}^n, g_\tau)$ and
$Du$ is a diffeomorphism from $\Omega$ to $\tilde{\Omega}$. For $\kappa\equiv 0$, finding a minimal Lagrangian diffeomorphism between two uniformly convex bounded domains in $(\mathbb{R}^n\times\mathbb{R}^n, g_\tau)$ is equivalent to solving (\ref{equ0.1.5}) with second boundary condition
\begin{equation}\label{equ0.1.6}
D u(\Omega)=\tilde{\Omega},
\end{equation}
that is,
\begin{equation}\label{e1.1.1}
\left\{ \begin{aligned}F_\tau(\lambda(D^2 u))&=c,\ \ x\in \Omega, \\
Du(\Omega)&=\tilde{\Omega}.
\end{aligned} \right.
\end{equation}
Here $Du$ is a minimal Lagrangian diffeomorphism from $\Omega$ to $\tilde \Omega$ in $(\mathbb{R}^n\times\mathbb{R}^n, g_\tau)$.

In dimension 2, Delano\"{e} \cite{P} obtained a unique smooth solution for the second boundary value problem of the Monge-Amp\`{e}re equation for $\tau=0$ in (\ref{e1.1.1}) if both domains are uniformly convex. Later the generalization of Delano\"{e}'s theorem to higher dimensions was given by Caffarelli \cite{L} and Urbas \cite{JU}. Using the parabolic method, Schn\"{u}rer and Smoczyk \cite{OK} also obtained the existence of solutions to (\ref{e1.1.1}) for $\tau=0$.

As far as $\tau=\frac{\pi}{2}$ is concerned, Brendle and Warren \cite{SM} proved the existence and uniqueness of the solution by the elliptic method, and the second author \cite{HR} obtained the existence of solution by the parabolic method. Then by the elliptic and parabolic method, the second author with Ou \cite{HO}, Ye \cite{HRY} \cite{CHY} and Chen \cite{CHY} proved the existence and uniqueness of the solution to (\ref{e1.1.1}) for $0<\tau<\frac{\pi}{2}$.

For a smooth function $f$, Chen, Zhang and the third author \cite{BCZ} proved that if $u$ satisfies
\begin{equation}\label{equ0.1.8}
\sum_{i=1}^{n}\arctan\lambda_{i}(D^2 u)=f(x),
\end{equation}
then $Df(x)$ is the mean curvature of the gradient graph $(x,Du(x))$ in $(\mathbb{R}^n\times\mathbb{R}^n, \delta_0)$. Motivated by the works of \cite{MW}and \cite{BCZ}, for $\tau\in \left[0,\frac{\pi}{2}\right]$ we obtain a generalization of their results.

\begin{Proposition}\label{prop2.10.1}
If $u$ satisfies
\begin{equation}\label{equ2.10.1}
F_\tau(\lambda(D^2 u))=f(x),
\end{equation}
 Then $Df(x)$ is the mean curvature of the gradient graph $(x,Du(x))$ in $(\mathbb{R}^n\times\mathbb{R}^n, g_\tau)$.
\end{Proposition}

For $f(x)=\kappa\cdot x+c$, Proposition \ref{prop2.10.1} becomes
\begin{Corollary}\label{c1.1}
If $u$ satisfies (\ref{e1.1}), then $\kappa$ is the constant mean curvature of the gradient graph $(x,Du(x))$ in $(\mathbb{R}^n\times\mathbb{R}^n, g_\tau)$.
\end{Corollary}

By the continuity method, for the second boundary value problem (\ref{e1.1}), we have
\begin{theorem}\label{t1.1}
For $\tau\in\left(0,\frac{\pi}{2}\right]$, if $|\kappa|$ is sufficiently small, then there exist a uniformly convex solution $u\in C^{\infty}(\bar{\Omega})$ and a unique constant $c$ solving (\ref{e1.1}), and $u$ is unique up to a constant.
\end{theorem}

\begin{rem}
For $\tau=0$, the problem (\ref{e1.1}) was solved by J. Urbas \cite{JU}, O.C. Schn\"{u}rer and K. Smoczyk \cite{OK}.
\end{rem}
The geometric meaning of this theorem is that if $\Omega$ and $\tilde{\Omega}$ are two uniformly convex bounded domains with smooth boundary in $\mathbb{R}^{n}$, then there exists a diffeomorphism $\psi=Du:\Omega\rightarrow \tilde{\Omega}$ such that
$$\Sigma:=\{(x,\psi(x)):x\in\Omega\}$$
is a Lagrangian submanifold, of which the mean curvature is $\kappa$ in $(\mathbb{R}^n\times\mathbb{R}^n,g_\tau)$.

In another paper, we shall point out by the parabolic method that $c$ means the coefficient of time variable for the translating solution of the parabolic problem corresponding to (\ref{e1.1}), and therefore $c$ can not be given in advance.

Theorem \ref{t1.1} presents an extension of the previous work on $\kappa=0$ done by Brendle-Warren \cite{SM}, Huang \cite{HR}, Huang-Ou \cite{HO}, Huang-Ye \cite{HRY} and Chen-Huang-Ye \cite{CHY}.

The rest of this article is organized as follows. In Section 2, we give the proof of  Proposition \ref{prop2.10.1} and introduce a class of fully nonlinear elliptic equation containing (\ref{e1.1}). Then we present a theorem on the corresponding second boundary value problem, which is a generalization of Theorem \ref{t1.1}. To prove this new theorem, we verify the strictly oblique estimate in Section 3, present the $C^2$ estimate in Section 4 and give the proof of this theorem by the continuity method in Section 5.

Throughout the following, Einstein's convention of summation over repeated indices will be adopted. We denote, for a smooth function $u$,
 $$u_{i}=\dfrac{\partial u}{\partial x_{i}},\ u_{ij}=\dfrac{\partial^{2}u}{\partial x_{i}\partial x_{j}},\ u_{ijk}=\dfrac{\partial^{3}u}{\partial x_{i}\partial x_{j}\partial
x_{k}}, \cdots .$$

\section{A generalization of Theorem \ref{t1.1}}

We begin with the proof of  Proposition \ref{prop2.10.1}.
\begin{proof}
Let $e_i=(0,\cdots,1,\cdots, 0)$ be the $i$-th axis vector in $\mathbb{R}^{n}\times\mathbb{R}^{n}$, $i=1,2,\cdots, 2n$. Then the tangential vector fields of $\overrightarrow{E}:=(x,Du(x))$---which is an $n$-dimensional submanifold of $ (\mathbb{R}^{n}\times\mathbb{R}^{n}, g_\tau)$---are
$$E_i=e_i+u_{ij}e_{n+j},\ \ i=1,\cdots,n.$$
Therefore,
$$\overline{\nabla}_{E_j}^{E_i}=u_{ijk}e_{n+k},$$
where $\overline{\nabla}$ is the Levi-Civita connection of $\mathbb{R}^{2n}$. Then the induced metric on $\overrightarrow{E}$ is given by
$$g_{ij}=\langle E_i,E_j\rangle=\langle e_i+u_{ik}e_{n+k},e_j+u_{jl}e_{n+l}\rangle=\sin\tau(\delta_{ij}+u_{ik}u_{kj})+2\cos\tau u_{ij}.$$
Denote $(g^{ij})=(g_{ij})^{-1}$, then the normal vector fields of $\overrightarrow{E}$ are
$$M_\alpha=M_\alpha^ie_i+\tilde M_\alpha^j e_{n+j},\ \ \alpha=1,\cdots,n.$$
Consequently, the part of the mean curvature vector
$$\overrightarrow{H}=g^{ij}\left(\overline{\nabla}_{E_i}^{E_j}\right)^\bot$$
on $M_\alpha$ is
\begin{equation}\label{equ2.10.2}
\begin{split}
H_\alpha:&=\langle\overrightarrow{H},M_\alpha\rangle\\
&=g^{ij}\langle\overline{\nabla}_{E_i}^{E_j},M_\alpha\rangle\\
&=g^{ij}\langle u_{ijk}e_{n+k},M_\alpha^l e_l+\tilde M_\alpha^p e_{n+p}\rangle\\
&=g^{ij}\left(\cos\tau u_{ijk}M_\alpha^k+\sin\tau u_{ijk}\tilde M_\alpha^k\right).
\end{split}
\end{equation}
Let $\overrightarrow{F}=\left(D f\right)^\bot$, then
\begin{equation}\label{equ2.10.3}
\begin{split}
F_\alpha:&=\langle D f^{\perp},M_\alpha\rangle\\
&=\langle f_le_{n+l},M_\alpha^i e_i+\tilde M_\alpha^j e_{n+j}\rangle\\
&=\cos\tau f_lM_\alpha^l+\sin\tau f_l\tilde M_\alpha^l.
\end{split}
\end{equation}
Comparing (\ref{equ2.10.2}) with (\ref{equ2.10.3}), we know that if for any $k=1,\cdots, n$ we have
\begin{equation}\label{equ2.10.4}
g^{ij}u_{ijk}=f_k,
\end{equation}
then $\overrightarrow{H}=\left(D f\right)^\bot$.

In conclusion, we complete the proof of Proposition \ref{prop2.10.1}.
\end{proof}
In order to prove Theorem \ref{t1.1} in four cases of all together, we reduce it to a more general form.
For the convenience, we introduce some notations.

It is obvious that $F_\tau(\lambda_1, \cdots, \lambda_n)$, $\tau\in\left(0,\frac{\pi}{2}\right]$ is a smooth symmetric function defined on ${\Gamma}^+_n$, where
$$\Gamma^+_n:=\left\{(\lambda_1,\cdots,\lambda_n)\in \mathbb{R}^n:\lambda_i>0,\ i=1,\cdots,n\right\}.$$
By direct calculation, we get
\begin{equation*}
F_\tau(0,\cdots,0)=\left\{ \begin{aligned}
& \frac{n\sqrt{a^2+1}}{2b}\ln\frac{a-b}{a+b},  &&0<\tau<\frac{\pi}{4},\\
& -\sqrt{2}n, &&\tau=\frac{\pi}{4},\\
& \frac{n\sqrt{a^2+1}}{b}\arctan\frac{a-b}{a+b}, \quad \quad &&\frac{\pi}{4}<\tau<\frac{\pi}{2},\\
& 0, \,\,&&\tau=\frac{\pi}{2},
\end{aligned} \right.
\end{equation*}
\begin{equation*}
F_\tau(+\infty,\cdots,+\infty)=\left\{ \begin{aligned}
& 0,  &&0<\tau<\frac{\pi}{4},\\
& 0, &&\tau=\frac{\pi}{4},\\
& \frac{n\pi\sqrt{a^2+1}}{4b}, \quad \quad &&\frac{\pi}{4}<\tau<\frac{\pi}{2},\\
& \frac{n\pi}{2}, \,\,&&\tau=\frac{\pi}{2},
\end{aligned} \right.
\end{equation*}
\begin{equation*}
\frac{\partial F_{\tau}}{\partial \lambda_i}=\left\{ \begin{aligned}
& \frac{\sqrt{a^2+1}}{(\lambda_i+a)^2-b^2}, \ \ \ \  &&0<\tau<\frac{\pi}{4},\\
& \frac{\sqrt{2}}{(1+\lambda_{i})^2}, &&\tau=\frac{\pi}{4},\\
& \frac{\sqrt{a^2+1}}{(\lambda_i+a)^2+b^2}, &&\frac{\pi}{4}<\tau<\frac{\pi}{2},\\
& \frac{1}{1+\lambda^2_{i}},  &&\tau=\frac{\pi}{2},
\end{aligned} \right.
\end{equation*}
and
\begin{equation*}
\frac{\partial^2 F_{\tau}}{\partial \lambda_i \partial \lambda_j}=\left\{ \begin{aligned}
& -\frac{2\sqrt{a^2+1}(\lambda_j+a)\delta_{ij}}{\left[(\lambda_i+a)^2-b^2\right]^2}, \ \ \ \  &&0<\tau<\frac{\pi}{4},\\
& -\frac{2\sqrt{2}\delta_{ij}}{(1+\lambda_{i})^3}, &&\tau=\frac{\pi}{4},\\
& -\frac{2\sqrt{a^2+1}(\lambda_j+a)\delta_{ij}}{\left[(\lambda_i+a)^2+b^2\right]^2}, &&\frac{\pi}{4}<\tau<\frac{\pi}{2},\\
& -\frac{2\lambda_j\delta_{ij}}{\left(1+\lambda^2_{i}\right)^2},  &&\tau=\frac{\pi}{2},
\end{aligned} \right.
\end{equation*}
for $i,j=1,\cdots,n$. Then
\begin{equation}\label{e2.1.1.1}
-\infty<F_\tau(0,\cdots,0)< F_\tau(+\infty,\cdots,+\infty)<+\infty,\ \ \tau\in\left(0,\frac{\pi}{2}\right],
\end{equation}
\begin{equation}\label{e2.1.2}
\frac{\partial F_\tau}{\partial \lambda_i}>0,\ \ 1\leq i\leq n \ \  \text{on}\ \  \Gamma^+_n,
\end{equation}
and
\begin{equation}\label{e2.1.3}
\left(\frac{\partial^2 F_\tau}{\partial \lambda_i\partial \lambda_j}\right)\leq 0 \ \  \text{on}\ \ \Gamma^+_n.
\end{equation}

For any $s_1>0$, $s_2>0$, define
$$\Gamma^{+}_{]s_1,s_2[}=\{(\lambda_{1},\cdots, \lambda_{n})\in {\Gamma}^+_n:0\leq\min_{1\leq i\leq n}\lambda_i\leq s_1,\ \max_{1\leq i\leq n}\lambda_i\geq s_2\}.$$
Then for any $(\lambda_{1},\cdots, \lambda_{n})\in \Gamma^{+}_{]s_{1},s_{2}[}$, we have
\begin{equation}\label{e2.1.4}
\sum_{i=1}^n\frac{\partial F_{\tau}}{\partial \lambda_i}\in\left\{ \begin{aligned}
&\left[\frac{\sqrt{a^2+1}}{(s_{1}+a)^2-b^2}, \frac{n\sqrt{a^2+1}}{a^2-b^2}\right], \ \ \ \  &&0<\tau<\frac{\pi}{4},\\
& \left[\frac{\sqrt{2}}{(1+s_{1})^2},n\sqrt{2}\right], &&\tau=\frac{\pi}{4},\\
& \left[\frac{\sqrt{a^2+1}}{(s_{1}+a)^2+b^2},\frac{n\sqrt{a^2+1}}{a^2+b^2}\right], &&\frac{\pi}{4}<\tau<\frac{\pi}{2},\\
&\left[\frac{1}{1+s_{1}^2} ,n\right],  &&\tau=\frac{\pi}{2},
\end{aligned} \right.
\end{equation}
and
\begin{equation}\label{e2.1.6}
\sum_{i=1}^n\frac{\partial F_{\tau}}{\partial \lambda_i}\lambda_i^2\in\left\{ \begin{aligned}
&\left[\frac{s_2^2\sqrt{a^2+1}}{(s_2+a)^2-b^2}, n\sqrt{a^2+1}\right], \ \ \ \  &&0<\tau<\frac{\pi}{4},\\
&\left[\frac{s_2^2\sqrt{2}}{(1+s_2)^2}, n\sqrt{2}\right], &&\tau=\frac{\pi}{4},\\
&\left[\frac{s_2^2\sqrt{a^2+1}}{(s_2+a)^2+b^2}, n\sqrt{a^2+1}\right], &&\frac{\pi}{4}<\tau<\frac{\pi}{2},\\
&\left[\frac{s_2^2}{1+s_2^2}, n\right],  &&\tau=\frac{\pi}{2}.
\end{aligned} \right.
\end{equation}

For any $(\mu_1,\cdots,\mu_n)\in {\Gamma}^+_n$, denote
$$\lambda_i=\frac{1}{\mu_i},\ \ 1\leq i\leq n,$$
and
$$\tilde F_\tau (\mu_1,\cdots,\mu_n):=-F_\tau(\lambda_1,\cdots,\lambda_n).$$
Then
\begin{equation*}
\frac{\partial \tilde F_\tau}{\partial \mu_{i}}=\lambda^{2}_{i}\frac{\partial F_\tau}{\partial \lambda_{i}},\quad \mu^{2}_{i}\frac{\partial \tilde F_\tau}{\partial \mu_{i}}=\frac{\partial F_\tau}{\partial \lambda_{i}},
\end{equation*}
and
\begin{equation*}
\begin{aligned}
\frac{\partial^2 \tilde F_{\tau}}{\partial \mu_i \partial \mu_j}&&&=-\lambda^3_i\left(\lambda_i\frac{\partial^2 F_\tau}{\partial \lambda^2_i}+2\frac{\partial F_\tau}{\partial \lambda_i}\right)\delta_{ij}\\
&&&=\left\{ \begin{aligned}
& -\frac{2\sqrt{a^2+1}(\mu_{i}+a)}{\left[(1+a\mu_{i})^2-(b\mu_{i})^2\right]^2}\delta_{ij}, \ \ \ \  &&0<\tau<\frac{\pi}{4},\\
& -\frac{2\sqrt{2}\delta_{ij}}{(1+\mu_{i})^3}, &&\tau=\frac{\pi}{4},\\
& -\frac{2\sqrt{a^2+1}(\mu_i+a)}{\left[(1+a\mu_i)^2+(b\mu_{i})^2\right]^2}\delta_{ij}, &&\frac{\pi}{4}<\tau<\frac{\pi}{2},\\
& -\frac{2\mu_i\delta_{ij}}{\left(1+\mu^2_{i}\right)^2},  &&\tau=\frac{\pi}{2}.
\end{aligned} \right.
\end{aligned}
\end{equation*}
Therefore, we have
\begin{equation*}
\frac{\partial \tilde F_\tau}{\partial \mu_i}>0,\ \ 1\leq i\leq n\ \  \text{on}\ \  \Gamma^+_n,
\end{equation*}
and
\begin{equation}\label{e2.1.11}
\left(\frac{\partial^2 \tilde F_\tau}{\partial \mu_i\partial \mu_j}\right)\leq 0\ \  \text{on}\ \  \Gamma^+_n.
\end{equation}

Motivated by (\ref{e2.1.1.1})-(\ref{e2.1.11}), in order to prove Theorem \ref{t1.1}, we introduce a class of nonlinear functions containing $F_\tau(\lambda)$, $\tau\in (0,\frac{\pi}{2}]$. Let $F(\lambda_{1},\cdots, \lambda_{n})$ be a smooth symmetric function defined on ${\Gamma}^+_n$, and satisfy
\begin{equation}\label{e1.2.0}
-\infty<F(0,\cdots,0)<F(+\infty,\cdots,+\infty)<+\infty,
\end{equation}
\begin{equation}\label{e1.2}
\frac{\partial F}{\partial \lambda_i}>0,\ \ 1\leq i\leq n\ \  \text{on}\ \  \Gamma^+_n,
\end{equation}
and
\begin{equation}\label{e1.2.1}
\left(\frac{\partial^2 F}{\partial \lambda_i\partial \lambda_j}\right)\leq 0\ \  \text{on}\ \  \Gamma^+_n.
\end{equation}
For any $(\mu_1,\cdots,\mu_n)\in {\Gamma}^+_n$, denote
$$\lambda_i=\frac{1}{\mu_i},\ \ 1\leq i\leq n,$$
and
$$\tilde F (\mu_1,\cdots,\mu_n):=-F(\lambda_1,\cdots,\lambda_n).$$
Assume that
\begin{equation}\label{e1.2.2}
\left(\frac{\partial^2 \tilde F}{\partial \mu_i\partial \mu_j}\right)\leq 0\ \  \text{on}\ \  \Gamma^+_n.
\end{equation}
In addition, for $s_1,s_2>0$, we assume that there exist positive constants $\Lambda_1$ and $\Lambda_2$, depending  on $s_1$ and $s_2$, such that for any $(\lambda_{1},\cdots, \lambda_{n})\in \Gamma^{+}_{]s_1,s_2[}$,
\begin{equation}\label{e1.3}
 \Lambda_1\leq\sum^{n}_{i=1}\frac{\partial F}{\partial \lambda_{i}}\leq \Lambda_2,
\end{equation}
and
\begin{equation}\label{e1.4}
  \Lambda_1\leq\sum^{n}_{i=1}\frac{\partial F}{\partial \lambda_{i}}\lambda^{2}_{i}\leq \Lambda_2.
\end{equation}

\begin{rem}
Since
$$\frac{\partial^2 \tilde F}{\partial \mu_i\partial \mu_j}=-\frac{\partial^2 F}{\partial \lambda_i\partial \lambda_j}\lambda_i^2\lambda_j^2-2\lambda_i^3\delta_{ij}\frac{\partial F}{\partial \lambda_i},$$
we cannot deduce (\ref{e1.2.2}) from (\ref{e1.2}) and (\ref{e1.2.1}).
\end{rem}

By the discussion above, we have
\begin{Proposition}\label{prop2.10.2}
The operator $F_\tau(\lambda)$, $\tau\in (0,\frac{\pi}{2}]$ satisfies the structure conditions (\ref{e1.2.0})-(\ref{e1.4}).
\end{Proposition}

For $f(x)\in C^{\infty}(\bar{\Omega})$, we define
$$\mathop{\operatorname{osc}}_{\bar{\Omega}}(f):=\max_{x,y\in\bar{\Omega}}|f(x)-f(y)|,$$
and
$${\mathscr{A}}_\delta:=\left\{f(x)\in C^{\infty}(\bar{\Omega}): f\ \text{is concave},\ \mathop{\operatorname{osc}}_{\bar{\Omega}}(f) \leq\delta \right\}.$$
 The constant
 $\delta$ is any positive constant satisfying
 $$\delta<\min\left\{F(+\infty,\cdots,+\infty)-F(\Theta_0,+\infty,\cdots,+\infty),F(0,\cdots,0,\Theta_0)-F(0,\cdots,0)\right\},$$
where $\Theta_{0}:=\left(\frac{|\tilde{\Omega}|}{|\Omega|}\right)^{1/n}$.

Considering the more general second boundary value problem
\begin{equation}\label{e1.8}
\left\{ \begin{aligned}
F\left(\lambda (D^2u)\right)&=f(x)+c,\ \ x\in \Omega, \\
Du(\Omega)&=\tilde{\Omega},
\end{aligned} \right.
\end{equation}
we claim that
\begin{theorem}\label{t1.2}
Let $F$ satisfy the structure conditions (\ref{e1.2.0})-(\ref{e1.4}) and $f\in {\mathscr{A}}_\delta$. If $|Df|$ is sufficiently small, then there exist a uniformly convex solution $u\in  C^{\infty}(\bar{\Omega})$ and a unique constant $c$ solving (\ref{e1.8}), and $u$ is unique up to a constant.
\end{theorem}
\begin{rem}\label{r2.4}
 It's not hard to deduce that if $|Df|$ is sufficiently small and $f$ is concave, then $f\in {\mathscr{A}}_\delta$.
\end{rem}
\begin{rem}\label{r2.5}
Let $n\geq 2$,  $f=\kappa_0 x_{1}$, $\Omega=\{x\in\mathbb{R}^n:|x|<1\}$, and $\tilde{\Omega}$ be a uniformly convex domain in $\mathbb{R}^{n}$. If $(u,c)$ satisfies
\begin{equation*}
\left\{ \begin{aligned}
\sum^n_{i=1}\arctan\lambda_{i}&=f(x)+c,\ \ x\in \Omega, \\
Du(\Omega)&=\tilde{\Omega}.
\end{aligned} \right.
\end{equation*}
Then $$|Df|=|\kappa_0|=\frac{1}{n}\mathop{\operatorname{osc}}_{\bar{\Omega}}(f+c)=\frac{1}{n}
\mathop{\operatorname{osc}}_{\bar{\Omega}}\left(\sum^n_{i=1}\arctan\lambda_{i}\right)<\frac{\pi}{2}.$$
Therefore, the range of $|Df|$  should be limited for the solvability of the problem  (\ref{e1.8}).
\end{rem}
By Proposition \ref{prop2.10.2} we know that, once Theorem \ref{t1.2} is proved, Theorem \ref{t1.1} holds.

In the next three sections, we are going to prove Theorem \ref{t1.2} through the continuity method, which is based on the strictly oblique estimate and the $C^2$ estimate.

\section{The strict obliqueness estimate}

To prove the strict obliqueness estimate, first we need
\begin{lemma}\label{l2.1}
Let $\lambda_{1}(x)$, $\cdots$, $\lambda_{n}(x)$ be the eigenvalues of $D^{2}u$ at $x$. Suppose that (\ref{e1.2.0}) and (\ref{e1.2}) hold, if $\mathop{\operatorname{osc}}_{\bar{\Omega}}(f) \leq\delta$ and $u\in  C^{\infty}(\bar{\Omega})$ is a uniformly convex solution of (\ref{e1.8}), then there exists $\mu>0$ and $\omega>0$ depending only on $F$, $\Theta_{0}$ and $\delta$ such that
\begin{equation}\label{e2.2}
\min_{1\leq i\leq n}\lambda_i(x)\leq\mu,\ \  \max_{1\leq i\leq n}\lambda_i(x)\geq\omega.
\end{equation}
\end{lemma}
\begin{proof}
By $Du(\Omega)=\tilde{\Omega}$, we have
$$\int_{\Omega}\det D^{2}u(x)dx=|\tilde{\Omega}|.$$
Then we can find $\bar x\in \bar{\Omega}$ such that
$$ \prod_{i=1}^n\lambda_i(\bar x)=\det D^{2}u(\bar x)=\frac{|\tilde{\Omega}|}{|\Omega|}=\Theta^{n}_{0}.$$
Therefore,
$$ \min_{1\leq i\leq n}\lambda_i(\bar x)\leq\Theta_{0}\leq\max_{1\leq i\leq n}\lambda_i(\bar x).$$
By (\ref{e1.8}) and condition (\ref{e1.2}), we obtain
\begin{equation*}
\begin{aligned}
F\left(\min_{1\leq i\leq n}\lambda_i(x),\cdots, \min_{1\leq i\leq n}\lambda_i(x)\right)&\leq F\left(\lambda_1(x),\cdots,\lambda_n(x)\right)\\
&=F\left(\lambda_1(\bar x),\cdots,\lambda_n(\bar x)\right)+f(x)-f(\bar x)\\
&\leq F\left(\Theta_0,+\infty,\cdots,+\infty\right)+\mathop{\operatorname{osc}}_{\bar{\Omega}}(f)\\
&\leq F\left(\Theta_0,+\infty,\cdots,+\infty\right)+\delta\\
&< F\left(+\infty,\cdots,+\infty\right),
\end{aligned}
\end{equation*}
and
\begin{equation*}
\begin{aligned}
F\left(\max_{1\leq i\leq n}\lambda_i(x),\cdots, \max_{1\leq i\leq n}\lambda_i(x)\right)&\geq F\left(\lambda_1(x),\cdots,\lambda_n(x)\right)\\
&=F\left(\lambda_1(\bar x),\cdots,\lambda_n(\bar x)\right)+f(x)-f(\bar x)\\
&\geq F(0,\cdots,0,\Theta_0)-\mathop{\operatorname{osc}}_{\bar{\Omega}}(f)\\
&\geq F(0,\cdots,0,\Theta_0)-\delta\\
&> F(0,\cdots,0).
\end{aligned}
\end{equation*}
By the monotonicity of $F$ and condition (\ref{e1.2.0}), we get the desired result.
\end{proof}

By Lemma \ref{l2.1}, the points $(\lambda_{1},\cdots, \lambda_{n})$ are always in $ \Gamma^{+}_{]\mu,\omega[}$ under the problem (\ref{e1.8}). Then there exist $\Lambda_1>0$ and $\Lambda_2>0$ depending only on $F$, $\Theta_{0}$ and $\delta$, such that $F$ satisfies the structure conditions (\ref{e1.3}) and (\ref{e1.4}). In the following, we always assume that $\Lambda_1$ and $\Lambda_2$ are universal constants depending only on the known data.

For technical needs below, we introduce the Legendre transformation of $u$. For any $x\in \mathbb{R}^n$, define
\begin{equation*}
\tilde{x}_{i}:=\frac{\partial u}{\partial x_{i}}(x),\ \ i=1,2,\cdots,n,
\end{equation*}
and
$$\tilde u(\tilde{x}_{1},\cdots,\tilde{x}_{n}):=\sum_{i=1}^{n}x_{i}\frac{\partial u}{\partial x_{i}}(x)-u(x).$$
In terms of $\tilde{x}_{1}$, $\cdots$, $\tilde{x}_{n}$ and $\tilde u(\tilde{x}_{1},\cdots,\tilde{x}_{n})$, we can easily check that
$$\left(\frac{\partial^{2} \tilde{u}}{\partial \tilde{x}_{i}\partial
\tilde{x}_{j}}\right)=\left(\frac{\partial^{2} u}{\partial x_{i}
\partial x_{j}}\right)^{-1}.$$
Let $\mu_{1}$, $\cdots$, $\mu_{n}$ be the eigenvalues of $D^{2}\tilde{u}$ at $\tilde{x}=D u(x)$. We denote
$$\mu_{i}=\lambda^{-1}_{i},\ \ i=1,2,\cdots,n.$$
Then
$$\frac{\partial \tilde F}{\partial \mu_{i}}=\lambda^{2}_{i}\frac{\partial F}{\partial \lambda_{i}},\quad \mu^{2}_{i}\frac{\partial \tilde F}{\partial \mu_{i}}=\frac{\partial F}{\partial \lambda_{i}}.$$
Moreover, it follows from (\ref{e1.8}) that
\begin{equation}\label{e2.7}
\left\{ \begin{aligned}\tilde{F}\left(\lambda(D^2\tilde{u})\right) &=-f(D\tilde{u})-c,\ \  \tilde{x}\in \tilde{\Omega}, \\
D\tilde{u}(\tilde{\Omega})&=\Omega.
\end{aligned} \right.
\end{equation}

\begin{rem}\label{r2.2}
By Lemma \ref{l2.1}, if $u$ is a smooth uniformly convex solution of (\ref{e1.8}), then the eigenvalues of $D^{2}u$ and $D^{2}\tilde{u}$ must be in $\Gamma^{+}_{]\mu,\omega[}$ and $\Gamma^{+}_{]\omega^{-1},\mu^{-1}[}$ respectively. Therefore, $\tilde F$ also satisfies the structure conditions (\ref{e1.3}) and (\ref{e1.4}).
\end{rem}

Next, we will carry out the strictly oblique estimate. Let $\mathscr{P}_n$ be the set of positive definite symmetric $n\times n$ matrices, and $\lambda_{1}(A)$, $\cdots$, $\lambda_{n}(A)$ be the eigenvalues of $A$.  For $A=(a_{ij})\in \mathscr{P}_n$, denote
$$F[A]:=F\left(\lambda_{1}(A),\cdots, \lambda_{n}(A)\right)$$
and
$$\left(a^{ij}\right)=(a_{ij})^{-1},\,\,\,
F^{ij}=\frac{\partial F}{\partial a_{ij}},\,\,\,F^{ij,rs}=\frac{\partial^{2} F}{\partial a_{ij}\partial a_{rs}}.$$

\begin{deff}
A smooth function $h:\mathbb{R}^n\rightarrow\mathbb{R}$ is called the defining function of $\tilde{\Omega}$ if
$$\tilde{\Omega}=\{p\in\mathbb{R}^{n} : h(p)>0\},\quad |Dh|_{{\partial\tilde{\Omega}}}=1,$$
and there exists $\theta>0$ such that for any $p=(p_{1},\cdots, p_{n})\in \tilde{\Omega}$ and $\xi=(\xi_{1}, \cdots, \xi_{n})\in \mathbb{R}^{n}$,
$$\frac{\partial^{2}h}{\partial p_{i}\partial p_{j}}\xi_{i}\xi_{j}\leq -\theta|\xi|^{2}.$$
\end{deff}

Therefore, the diffeomorphism condition $Du(\Omega)=\tilde{\Omega}$ in (\ref{e1.8}) is equivalent to
\begin{equation}\label{e2.10}
 h(Du)=0,\ \ x\in \partial\Omega.
\end{equation}
Then (\ref{e1.8}) can be rewritten as
\begin{equation}\label{e2.101}
\left\{ \begin{aligned}
F\left[D^2 u\right]&=f(x)+c,\ \ &&x\in \Omega, \\
h(Du)&=0,             &&x\in\partial \Omega.
\end{aligned} \right.
\end{equation}
This is an oblique boundary value problem of second order fully nonlinear elliptic equation. We also denote $\beta=(\beta^{1}, \cdots, \beta^{n})$ with $\beta^{i}:=h_{p_{i}}(Du)$, and $\nu=(\nu_{1},\cdots,\nu_{n})$ as the unit inward normal vector at $x\in\partial\Omega$.
The expression of the inner product  is
\begin{equation*}
\langle\beta, \nu\rangle=\beta^{i}\nu_{i}.
\end{equation*}
\begin{lemma}\label{l1.3}(See J. Urbas \cite{JU}.) Let $\nu=(\nu_{1},\nu_{2}, \cdots,\nu_{n})$ be the unit inward normal vector of $\partial\Omega$. If $u\in C^{2}(\bar{\Omega})$  with $D^{2}u\geq0$,
then there holds $h_{p_{k}}(Du)\nu_{k}\geq0$.
\end{lemma}
Now, we can present
\begin{lemma}\label{l2.3}
Let $F$ satisfy the structure conditions (\ref{e1.2.0})-(\ref{e1.4}) and  $f\in {\mathscr{A}}_\delta$. If $u$ is a uniformly smooth convex solution of (\ref{e1.8}) and $|D f|$ is sufficiently small, then the strict obliqueness estimate
\begin{equation}\label{e2.11}
\langle\beta, \nu\rangle\geq \frac{1}{C_{1}}>0
\end{equation}
holds on $\partial \Omega$ for some universal constant $C_{1}$, which depends only on $F$, $\Theta_0$, $\Omega,$ $\tilde{\Omega}$ and $\delta$.
\end{lemma}
\begin{proof}
Define
$$v=\langle \beta,\nu\rangle+h(Du).$$
Let $x_0\in \partial \Omega$ such that
$$\langle \beta,\nu\rangle(x_0)=h_{p_k}(Du(x_0))\nu_k(x_0)=\min_{\partial\Omega}\langle \beta,\nu\rangle.$$
By rotation, we may assume that $\nu(x_0)=(0,\cdots,0,1)=:e_n$.
Using the above assumptions and the boundary condition, we obtain
$$v(x_0)=\min_{\partial\Omega} v=h_{p_{n}}(Du(x_{0})).$$
 By the convexity of $\Omega$ and its smoothness, we extend $\nu$ smoothly to a tubular neighborhood of $\partial\Omega$ such that in the matrix sense
\begin{equation}\label{eqq2.3}
  \left(\nu_{kl}\right):=\left(D_k\nu_l\right)\leq -\frac{1}{C}\operatorname{diag} (1,\cdots, 1,0),
\end{equation}
where $C$ is a positive constant. By Lemma \ref{l1.3}, we see that $h_{p_{n}}(Du(x_{0}))\geq0$.

At $x_0$ we have
\begin{equation}\label{eqq2.4}
 0=v_r=h_{p_np_k}u_{kr}+h_{p_k}\nu_{kr}+h_{p_k}u_{kr}, \quad 1\leq r\leq n-1.
\end{equation}
We assume that the following key estimate
\begin{equation}\label{eqq2.5}
 v_n(x_0)>-C
\end{equation}
holds which will be proved later, where $C$ is a constant depending only on $\Omega$, $h$, $\tilde{h}$ and $\delta$.

It's not hard to check that (\ref{eqq2.5}) can be rewritten as
\begin{equation}\label{eqq2.6}
 h_{p_np_k}u_{kn}+h_{p_k}\nu_{kn}+h_{p_k}u_{kn}>-C.
\end{equation}
Multiplying (\ref{eqq2.6}) with $h_{p_n}$ and (\ref{eqq2.4}) with $h_{p_r}$ respectively, and summing up together, we obtain
\begin{equation}\label{eqq2.7}
 h_{p_k}h_{p_l}u_{kl}\geq -Ch_{p_n}- h_{p_k}h_{p_l}\nu_{kl}- h_{p_k}h_{p_np_l}u_{kl}.
\end{equation}
Using (\ref{eqq2.3}) and
$$ 1\leq r\leq n-1,\quad h_{p_k}u_{kr}=\frac{\partial h(Du)}{\partial x_r}=0,\quad h_{p_k}u_{kn}=\frac{\partial h(Du)}{\partial x_n}\geq 0,\quad -h_{p_np_n}\geq 0,$$
we have
$$h_{p_k}h_{p_l}u_{kl}\geq-Ch_{p_n}+\frac{1}{C}|Dh|^2-\frac{1}{C}h^2_{p_n}=-Ch_{p_n}+\frac{1}{C}-\frac{1}{C}h^2_{p_n}.$$
For the last term of the above inequality, we distinguish two cases at $x_0$.

Case (i).  If
$$-Ch_{p_n}+\frac{1}{C}-\frac{1}{C}h^2_{p_n}\leq \frac{1}{2C},$$
then
$$h_{p_k}(Du)\nu_{k}=h_{p_n }\geq \sqrt{\frac{1}{2}+\frac{C^4}{4}}-\frac{C^2}{2}.$$
It shows that there is a uniform positive lower bound for the quantity $\min_{\partial\Omega}\langle \beta,\nu\rangle$.

Case (ii). If
$$-Ch_{p_n}+\frac{1}{C}-\frac{1}{C}h^2_{p_n}> \frac{1}{2C},$$
then we obtain a positive lower bound of $h_{p_k}h_{p_l}u_{kl}$.

Let $\tilde{u}$ be the Legendre transformation of $u$, then $\tilde{u}$ satisfies
\begin{equation}\label{eqq2.9}
\left\{ \begin{aligned}\tilde{F}\left[D^2\tilde{u}\right] &=-f(D\tilde{u})-c, \ \ && \tilde x\in\tilde{\Omega},\\
\tilde{h}(D\tilde{u})&=0, && \tilde x\in\partial\tilde{\Omega},
\end{aligned} \right.
\end{equation}
where $\tilde{h}$ is the defining function of $\Omega$. That is,
$$\Omega=\{\tilde{p}\in\mathbb{R}^{n} : \tilde{h}(\tilde{p})>0\},\ \ \ |D\tilde{h}|_{\partial\Omega}=1, \ \ \ D^2\tilde{h}\leq -\tilde{\theta}I,$$
where $\tilde{\theta}$ is some positive constant. The unit inward normal vector of $\partial\Omega$ can be expressed by $\nu=D\tilde{h}$. For the same reason,
$\tilde{\nu}=Dh$, where $\tilde{\nu}=(\tilde{\nu}_{1}, \tilde{\nu}_{2},\cdots,\tilde{\nu}_{n})$ is the unit inward normal vector of $\partial\tilde{\Omega}$.

Let $\tilde{\beta}=(\tilde{\beta}^{1}, \cdots, \tilde{\beta}^{n})$ with $\tilde{\beta}^{k}:=\tilde{h}_{p_{k}}(D\tilde{u})$. We also define
$$\tilde{v}=\langle\tilde{\beta}, \tilde{\nu}\rangle+\tilde{h}(D\tilde{u}),$$
in which
$$\langle\tilde{\beta}, \tilde{\nu}\rangle=\langle\beta, \nu\rangle.$$

Denote $\tilde{x}_{0}=Du(x_{0})$. Then $\tilde{v}(\tilde{x}_{0})=v(x_{0})=\min_{\partial\tilde{\Omega}} \tilde{v}$. Using the same methods, under the assumption of
\begin{equation}\label{eqq2.9.1.1}
\tilde{v}_{n}(\tilde{x}_{0})\geq -C,
\end{equation}
we obtain the positive lower bounds of $\tilde{h}_{p_{k}}\tilde{h}_{p_{l}}\tilde{u}_{kl}$, or
$$h_{p_{k}}(Du)\nu_{k}=\tilde{h}_{p_{k}}(D\tilde{u})\tilde{\nu}_{k}=\tilde{h}_{p_{n}}\geq\sqrt{\frac{1}{2}+\frac{C^4}{4}}-\frac{C^2}{2}.$$
We notice that
$$\tilde{h}_{p_{k}}\tilde{h}_{p_{l}}\tilde{u}_{kl}=\nu_{i}\nu_{j}u^{ij}.$$
Then by the positive lower bounds of $h_{p_{k}}h_{p_{l}}u_{kl}$ and $\tilde{h}_{p_{k}}\tilde{h}_{p_{l}}\tilde{u}_{kl}$, the lemma follows from
\begin{equation}\label{e2.11.1}
\langle \beta,\nu\rangle=\sqrt{h_{p_k}h_{p_l}u_{kl}u^{ij}\nu_i\nu_j},
\end{equation}
which is proved in \cite{JU}.

It remains to prove the key estimate (\ref{eqq2.5}) and (\ref{eqq2.9.1.1}). We prove (\ref{eqq2.5}) first. By $D^2\tilde{h}\leq -\tilde{\theta}I$ and (\ref{e1.3}) we have
\begin{equation}\label{eqq2.10}
L\tilde{h}\leq -\tilde{\theta}\sum^{n}_{i=1} F^{ii},
\end{equation}
where $L:=F^{ij}\partial_{ij}$. On the other hand,
\begin{equation}\label{eqq2.11}
 \begin{aligned}
Lv=&h_{p_kp_lp_m}\nu_kF^{ij}u_{li}u_{mj}+2h_{p_kp_l}F^{ij}\nu_{kj}u_{li}\\
&+h_{p_kp_l}F^{ij}u_{lj}u_{ki}+h_{p_kp_l}\nu_kLu_l+h_{p_k}L\nu_k+h_{p_k}Lu_k.
 \end{aligned}
\end{equation}
At first we estimate the first term on the right hand side of (\ref{eqq2.11}).  By the diagonal basis and (\ref{e1.4}), we have
$$|h_{p_kp_lp_m}\nu_kF^{ij}u_{li}u_{mj}|\leq C\sum_{i=1}^{n}\frac{\partial F}{\partial \lambda_i}\lambda_i^2\leq C, $$
where $C$ is a constant depending only on $h$, $\Omega$, $\Lambda_1$, $\Lambda_2$ and $\delta$. Similarly, we also get
$$ |h_{p_kp_l}F^{ij}u_{lj}u_{ki}|\leq C\sum_{i=1}^{n}\frac{\partial F}{\partial \lambda_i}\lambda_i^2\leq C.$$
For the second term, by Cauchy inequality, we obtain
\begin{equation*}
\begin{aligned}
|2h_{p_{k}p_{l}}F^{ij}\nu_{kj}u_{li}|&\leq C\sum^{n}_{i=1}\frac{\partial F}{\partial \lambda_{i}}\lambda_{i}
=C\sum^{n}_{i=1}\sqrt{\frac{\partial F}{\partial \lambda_{i}}}\sqrt{\frac{\partial F}{\partial \lambda_{i}}}\lambda_{i}\\
&\leq C\left(\sum^{n}_{i=1}\frac{\partial F}{\partial \lambda_{i}}\right)\left(\sum^{n}_{i=1}\frac{\partial F}{\partial \lambda_{i}}\lambda^{2}_{i}\right)\\
&\leq C.
\end{aligned}
\end{equation*}
By (\ref{e1.8}) we have $Lu_l=f_{l}$. Then we get
$$|h_{p_kp_l}\nu_kLu_l|\leq C,\quad |h_{p_k}Lu_k|\leq C.$$
It follows from (\ref{e1.3}) that
$$|h_{p_{k}}L\nu_k|\leq C\sum^{n}_{i=1} F^{ii}.$$
Inserting these into (\ref{eqq2.11}) and using (\ref{e1.3}), it is immediate to check that there exists a positive constant $C$ depending only on $h$, $\Omega$, $\Lambda_1$, $\Lambda_2$, $\Theta_0$ and $\delta$, such that
\begin{equation}\label{eqq2.12}
Lv\leq C\sum^{n}_{i=1} F^{ii}.
\end{equation}

Denote a neighborhood of $x_0$ in $\Omega$ by
$$\Omega_{\rho}:=\Omega\cap B_{\rho}(x_0),$$
where $\rho$ is a positive constant such that $\nu$ is well defined in $\Omega_{\rho}$.
To obtain the desired results, it suffices to consider the function

$$\Phi(x):=v(x)-v(x_0)+C_0\tilde{h}(x)+A|x-x_0|^2,$$
where $C_0$ and $A$ are positive constants to be determined. On $\partial\Omega$, it is clear that $\Phi\geq 0$. Since $v$ is bounded, we can choose $A$ large enough such that on $\Omega\cap \partial B_{\rho}(x_0)$
$$\Phi(x)=v(x)-v(x_0)+C_0\tilde{h}(x)+A\rho^2>0.$$
 It follows from (\ref{eqq2.10}) that
$$L(C_0\tilde{h}(x)+A|x-x_0|^2)\leq  (-C_0\tilde{\theta}+2A)\sum_{i=1}^n F^{ii}.$$
Then using (\ref{eqq2.12}) and choosing $C_0\gg A$  we have
$$L\Phi(x)\leq 0.$$
Therefore,
\begin{equation}\label{eqq2.13}
\left\{ \begin{aligned}
   L\Phi&\leq 0,\ \  &&x\in\Omega_{\rho},\\
   \Phi&\geq 0,\ \  &&x\in\partial\Omega_{\rho}.
                          \end{aligned} \right.
\end{equation}
We apply the maximum principle to get
$$\Phi|_{\Omega_{\rho}}\geq \min_{\partial\Omega_{\rho}}\Phi\geq 0.$$
Combining it with $\Phi(x_0)=0$, we obtain $\partial_n\Phi(x_0)\geq 0$, which gives the desired estimate (\ref{eqq2.5}).

Finally, we prove (\ref{eqq2.9.1.1}). The proof of (\ref{eqq2.9.1.1}) is similar to the one of (\ref{eqq2.5}).  Define
$$\tilde L:= \tilde F^{ij}\partial_{ij}+f_{p_i}\partial_i.$$
By (\ref{eqq2.9}), we see that $\tilde L\tilde u_l=0$, and thus
$$\tilde L\tilde v=\tilde F^{ij}\tilde u_{mj}\tilde u_{li}\tilde h_{p_kp_lp_m}\tilde \nu_k+2\tilde h_{p_kp_l}\tilde F^{ij}\tilde u_{li}\tilde \nu_{kj}+\tilde F^{ij}\tilde h_{p_k}\tilde \nu_{kij}+\tilde h_{p_kp_l}\tilde F^{ij}\tilde u_{lj}\tilde u_{ki}+\tilde h_{p_k}f_{p_i}\tilde \nu_{ki}.$$
By making use of  the following identities
\begin{equation*}
\frac{\partial \tilde F}{\partial \mu_{i}}=\lambda^{2}_{i}\frac{\partial F}{\partial \lambda_{i}},\quad \mu^{2}_{i}\frac{\partial \tilde F}{\partial \mu_{i}}=\frac{\partial F}{\partial \lambda_{i}}.
\end{equation*}
we deduce that  $\tilde{F}$ satisfies the structure conditions (\ref{e1.2.0})-(\ref{e1.4}). Repeating the proof of (\ref{eqq2.12}), we have
\begin{equation}\label{eqq2.13aa}
\tilde L\tilde v\leq C\sum^{n}_{i=1} \tilde F^{ii},
\end{equation}
where $C$ depends only on $\Omega$, $\tilde \Omega$, $\Theta_0$ and $\delta$.

Denote a neighborhood of $\tilde{x}_0$ in $\tilde\Omega$ by
$$\tilde\Omega_{r}:=\tilde\Omega\cap B_{r}(\tilde{x}_0),$$
where $r$ is a positive constant such that $\tilde\nu$ is well defined in $\tilde\Omega_{r}$. Consider
$$\tilde\Phi(y):=\tilde v(y)-\tilde v(\tilde{x}_0)+\tilde C_0 h(y)+\tilde A|y-\tilde{x}_0|^2,$$
where $\tilde C_0$ and $\tilde A$ are positive constants to be determined. It is clear that $\tilde\Phi\geq 0$ on $\partial\tilde \Omega$. Since $\tilde v$ is bounded, we can choose $\tilde A$ large enough such that on $\tilde\Omega\cap \partial B_{r}(\tilde{x}_0)$
$$\tilde\Phi(y)=\tilde v(y)-\tilde v(\tilde{x}_0)+\tilde C_0h(y)+\tilde A r^2>0.$$
By (\ref{e1.4}) and (\ref{eqq2.13aa}), it is not difficult to show that
$$\tilde L\tilde\Phi(y)\leq \left(C-\frac{\tilde C_0\theta}{2}+2\tilde A\right)\sum_{i=1}^n \tilde F^{ii}+2 \tilde A f_{p_i}(y_i-\tilde{x}_{0i})-\tilde C_0\left(\frac{\theta}{2}\sum_{i=1}^n \tilde F^{ii}-f_{p_i}\partial_i h\right). $$
In order to make
$$\tilde L\tilde\Phi(y)\leq 0,$$
we only need to choose $\tilde C_0\gg \tilde A$ and
$$|D f|\leq \frac{\theta\Lambda_1}{2}\cdot\frac{1}{\max_{\bar{\tilde \Omega}}|D h|}.$$
Consequently,
\begin{equation}\label{eqq2.13a}
\left\{ \begin{aligned}
   \tilde L\tilde \Phi&\leq 0,\ \  &&y\in\tilde \Omega_{r},\\
   \tilde\Phi&\geq 0,\ \  &&y\in\partial\tilde\Omega_{r}.
                          \end{aligned} \right.
\end{equation}
Therefore, we get (\ref{eqq2.9.1.1}) as same as  the argument in (\ref{eqq2.5}). Thus the proof of (\ref{e2.11}) is completed.
\end{proof}

\section{The $C^2$ estimate}
The following definition provides a basic connection between (\ref{eqq2.9}) and (\ref{e1.8}) and will be used frequently in the sequel.
\begin{deff}\label{d1.10}
We say that  $\tilde{u}$ in (\ref{eqq2.9})  is a dual solution to (\ref{e1.8}).
\end{deff}
We now proceed to carry out the $C^2$ estimate. The strategy is to  reduce the $C^2$ global estimate
of $u$ and $\tilde{u}$ to the boundary.

\begin{lemma}\label{lem3.1}
If $u$ is a smooth uniformly convex solution of (\ref{e1.8}) and there hold (\ref{e1.2}), (\ref{e1.2.1}) and (\ref{e1.3}), then there exists a positive constant $C$ depending only on $n$, $\Omega$, $\tilde{\Omega}$, $\Lambda_1$ and $\operatorname{diam}(\Omega)$, such that
\begin{equation}\label{eq3.1}
   \sup_{\Omega}|D^2u|\leq \max_{\partial\Omega} |D^2u|+C\sup_\Omega|D^2f|.
\end{equation}
\end{lemma}
\begin{proof}
Without loss of generality, we may assume that $\Omega$ lies in cube $[0,d]^{n}$. Let
$$L:=F^{ij}\partial_{ij}.$$
For any unit vector $\xi$, differentiating the equation in (\ref{e1.8}) twice in direction $\xi$ gives
\begin{equation*}
Lu_{\xi\xi}+F^{ij,rs}u_{ij\xi}u_{rs\xi}=f_{\xi\xi}.
\end{equation*}
Then by the concavity of $F$ on $\Gamma^+_n$, we have
\begin{equation}\label{e3.0}
Lu_{\xi\xi}=-F^{ij,rs}u_{ij\xi}u_{rs\xi}+f_{\xi\xi}\geq f_{\xi\xi}.
\end{equation}

Let $$v=\sup_{\partial\Omega}u_{\xi\xi}+\frac{1}{\Lambda_1}\left(ne^{ d}-\sum^{n}_{i=1}e^{ x_{i}}\right)\sup_{\Omega}|f_{\xi\xi}|.$$
By direct calculation and (\ref{e1.3}), we obtain
\begin{equation}\label{e3.1}
\begin{aligned}
Lv=&-\frac{1}{\Lambda_1}\sup_{\Omega}|f_{\xi\xi}|\left(\sum^{n}_{i=1}e^{ x_{i}}F^{ii}\right)\\
\leq &-\frac{1}{\Lambda_1}\sup_{\Omega}|f_{\xi\xi}|\left(\sum^{n}_{i=1}F^{ii}\right)\\
\leq &-\sup_{\Omega}|f_{\xi\xi}|.
\end{aligned}
\end{equation}
Combining (\ref{e3.0}) with (\ref{e3.1}), we have
$$L(v-u_{\xi\xi})\leq -\left(\sup_{\Omega}|f_{\xi\xi}|+f_{\xi\xi}\right)\leq 0.$$
It is obvious that $v-u_{\xi\xi}\geq 0$ on $\partial\Omega$. Then by the maximum principle we obtain
$$\sup_{\Omega}u_{\xi\xi}\leq \sup_{\Omega}v\leq \sup_{\partial\Omega}u_{\xi\xi}+\frac{ne^{d}}{\Lambda_1}\sup_{\Omega}|f_{\xi\xi}|.$$
This completes the proof of (\ref{eq3.1}).
\end{proof}

Next, we estimate the second order derivative on the boundary. By differentiating the boundary condition $h(Du)=0$ in any tangential direction $\varsigma$, we have
\begin{equation}\label{eq3.2}
   u_{\beta \varsigma}=h_{p_k}(Du)u_{k\varsigma}=0.
\end{equation}
The second order derivative of $u$ on the boundary is controlled by $u_{\beta \varsigma}$, $u_{\beta \beta}$ and $u_{\varsigma\varsigma}$. In the following we give the arguments as in \cite{JU}, one can see there for more details.

At $x\in \partial\Omega$, any unit vector $\xi$ can be written in terms of a tangential component $\varsigma(\xi)$ and a component in the direction $\beta$ by
$$\xi=\varsigma(\xi)+\frac{\langle \nu,\xi\rangle}{\langle\beta,\nu\rangle}\beta,$$
where
$$\varsigma(\xi):=\xi-\langle \nu,\xi\rangle \nu-\frac{\langle \nu,\xi\rangle}{\langle\beta,\nu\rangle}\beta^T,$$
and
$$\beta^T:=\beta-\langle \beta,\nu\rangle \nu.$$
By the strict obliqueness estimate (\ref{e2.11}), we have
\begin{equation}\label{eq3.3}
\begin{aligned}
|\varsigma(\xi)|^{2}&=1-\left(1-\frac{|\beta^{T}|^{2}}{\langle\beta,\nu\rangle^{2}}\right)\langle\nu,\xi\rangle^{2}
-2\langle\nu,\xi\rangle\frac{\langle\beta^{T},\xi\rangle}{\langle\beta,\nu\rangle}\\
&\leq 1+C\langle\nu,\xi\rangle^{2}-2\langle\nu,\xi\rangle\frac{\langle\beta^{T},\xi\rangle}{\langle\beta,\nu\rangle}\\
&\leq C.
\end{aligned}
\end{equation}
Denote $\varsigma:=\frac{\varsigma(\xi)}{|\varsigma(\xi)|}$, then by (\ref{eq3.3}) and (\ref{e2.11}) we obtain
\begin{equation}\label{eq3.4}
\begin{aligned}
u_{\xi\xi}&=|\varsigma(\xi)|^{2}u_{\varsigma\varsigma}+2|\varsigma(\xi)|\frac{\langle\nu,\xi\rangle}{\langle\beta,\nu\rangle}u_{\beta\varsigma}+
\frac{\langle\nu,\xi\rangle^{2}}{\langle\beta,\nu\rangle^{2}}
u_{\beta\beta}\\
&=|\varsigma(\xi)|^{2}u_{\varsigma\varsigma}+\frac{\langle\nu,\xi\rangle^{2}}{\langle\beta,\nu\rangle^{2}}
u_{\beta\beta}\\
&\leq C(u_{\varsigma\varsigma}+u_{\beta\beta}),
\end{aligned}
\end{equation}
where $C$ depends only on $\Omega$, $\tilde \Omega$, $\Lambda_1$, $\Lambda_2$, $\delta$ and the constant $C_1$ in (\ref{e2.11}). Therefore, we only need to estimate $u_{\beta\beta}$ and $u_{\varsigma\varsigma}$ respectively.

First we have
\begin{lemma}\label{lem3.2}
Let $F$ satisfy the structure conditions (\ref{e1.2.0})-(\ref{e1.4}) and $f\in {\mathscr{A}}_\delta$. If $u$ is a smooth uniformly convex solution of (\ref{e1.8}), then there exists a positive constant $C$ depending only on $\Omega$, $\tilde{\Omega}$, $\Lambda_1$, $\Lambda_2$ and $\delta$, such that
\begin{equation}\label{eq3.5}
   \max_{\partial\Omega}u_{\beta\beta} \leq C.
\end{equation}
\end{lemma}

\begin{proof}
Let $x_0\in\partial\Omega$ satisfy $u_{\beta\beta}(x_0)=\max_{\partial\Omega}u_{\beta\beta}$. Consider the barrier function
$$\Psi:=-h(Du)+C_0\tilde{h}+A|x-x_0|^2.$$
For any $x\in \partial\Omega$, $Du(x)\in \partial\tilde{\Omega}$, then $h(Du)=0$. It is clear that $\tilde{h}=0$ on $\partial\Omega$. As the proof of (\ref{eqq2.13}), we can find the constants $C_0$ and $A$ such that
\begin{equation}\label{eq3.6}
\begin{cases}
   L\Psi\leq 0,\quad x\in\Omega_{\rho},\\
   \Psi\geq 0 ,\quad\  x\in\partial\Omega_{\rho}.
\end{cases}
\end{equation}
By the maximum principle, we get
$$\Psi(x)\geq 0,\quad\quad x\in \Omega_\rho.$$
Combining it with $\Psi(x_0)=0$ we obtain $\Psi_{\beta}(x_0)\geq 0$, which implies
$$\frac{\partial h}{\partial \beta}(Du(x_0))\leq C_0.$$

On the other hand, we see that at $x_0$,
$$\frac{\partial h}{\partial \beta}=\langle Dh(Du),\beta\rangle=\frac{\partial h}{\partial p_k}u_{kl}\beta^l=\beta^ku_{kl}\beta^l=u_{\beta\beta}.$$
Therefore,
$$u_{\beta\beta}=\frac{\partial h}{\partial \beta}\leq C.$$
\end{proof}

Next, we estimate the double tangential derivative.
\begin{lemma}\label{lem3.3}
Let $F$ satisfy the structure conditions (\ref{e1.2.0})-(\ref{e1.4}) and $f\in {\mathscr{A}}_\delta$. If $u$ is a smooth uniformly convex solution of (\ref{e1.8}), then there exists a positive constant $C$ depending only on $\Omega$, $\tilde{\Omega}$, $\Lambda_1$, $\Lambda_2$ and $\delta$, such that
\begin{equation}\label{eq3.7}
   \max_{\partial\Omega}u_{\varsigma\varsigma} \leq C.
\end{equation}
\end{lemma}

\begin{proof}
Assume that $ u_{\varsigma\varsigma}|_{\partial\Omega}$ attains its maximum at $x_0\in\partial \Omega$. Let
$$M:=u_{11}(x_0)=\max_{\partial\Omega}u_{\varsigma\varsigma},$$
and $e_n$ be the unit inward normal vector of $\partial\Omega$ at $x_0$.

For any $x\in \partial\Omega$,  we have  by (\ref{eq3.3}),
\begin{equation}\label{eq3.8}
\begin{aligned}
u_{\xi\xi}&=|\varsigma(\xi)|^2u_{\varsigma\varsigma}+ \frac{\langle \nu,\xi\rangle^2}{\langle \beta,\nu\rangle^2}u_{\beta\beta}\\
          &\leq \left(1+C\langle \nu,\xi\rangle^2-2\langle \nu,\xi\rangle \frac{\langle \beta^T,\xi\rangle}{\langle \beta,\nu\rangle}\right)M
                 + \frac{\langle \nu,\xi\rangle^2}{\langle \beta,\nu\rangle^2}u_{\beta\beta}.
\end{aligned}
\end{equation}
Without loss of generality, we assume that $M\geq 1$. Then by (\ref{e2.11}) and (\ref{eq3.5}) we have
\begin{equation}\label{eq3.9}
  \frac{u_{\xi\xi}}{M}+2\langle \nu,\xi\rangle \frac{\langle \beta^T,\xi\rangle}{\langle \beta,\nu\rangle}
             \leq 1+C\langle \nu,\xi\rangle^2.
\end{equation}
Let $\xi=e_1$, then
\begin{equation}\label{eq3.10}
  \frac{u_{11}}{M}+2\langle \nu,e_1\rangle \frac{\langle \beta^T,e_1\rangle}{\langle \beta,\nu\rangle}
             \leq 1+C\langle \nu,e_1\rangle^2.
\end{equation}
We see that the function
\begin{equation}\label{eq3.11}
w:=A|x-x_0|^2-\frac{u_{11}}{M}-2\langle \nu,e_1\rangle \frac{\langle \beta^T,e_1\rangle}{\langle \beta,\nu\rangle}+C\langle \nu,e_1\rangle^2+1
\end{equation}
satisfies
$$w|_{\partial\Omega}\geq 0,\quad  w(x_0)=0.$$
Then, it follows by (\ref{eq3.1}) that we can choose the constant $A$ large enough such that
$$w|_{\Omega \cap \partial B_{\rho}(x_0)} \geq 0.$$

By (\ref{e1.8}), we deduce that
$$Lw\leq  C\sum^{n}_{i=1} F^{ii}.$$
As in the proof of Lemma \ref{lem3.2}, we consider the function
$$\Upsilon:=w+C_0\tilde{h}.$$
A standard barrier argument shows that
$$\Upsilon_{\beta}(x_0)\geq0.$$
Therefore,
\begin{equation}\label{eq3.12}
 u_{11\beta}(x_0)\leq CM.
\end{equation}
On the other hand, differentiating $h(Du)$ twice in the direction $e_1$ at $x_0$, we have
$$h_{p_k}u_{k11}+h_{p_kp_l}u_{k1}u_{l1}=0.$$
The concavity of $h$ yields that
$$h_{p_k}u_{k11}=-h_{p_kp_l}u_{k1}u_{l1}\geq \theta M^2.$$
Combining it with $h_{p_k}u_{k11}=u_{11\beta}$, and using (\ref{eq3.12}) we obtain
$$\theta M^2\leq CM.$$
Then we get the upper bound of $M=u_{11}(x_0)$ and thus the desired result follows.
\end{proof}

By Lemma \ref{lem3.2}, Lemma \ref{lem3.3} and (\ref{eq3.4}), we obtain the $C^2$ a-priori estimate on the boundary.
\begin{lemma}\label{lem3.4}
Let $F$ satisfy the structure conditions (\ref{e1.2.0})-(\ref{e1.4}) and $f\in {\mathscr{A}}_\delta$.  If $u$ is a smooth uniformly convex solution of (\ref{e1.8}), then there exists a positive constant $C$ depending only on $\Omega$, $\tilde{\Omega}$, $\Lambda_1$, $\Lambda_2$ and $\delta$, such that
\begin{equation}\label{eq3.13}
\max_{\partial\Omega}|D^2u| \leq C.
\end{equation}
\end{lemma}

In terms of Lemma \ref{lem3.1} and Lemma \ref{lem3.4}, we see that
\begin{lemma}\label{lem3.5}
Let $F$ satisfy the structure conditions (\ref{e1.2.0})-(\ref{e1.4}) and $f\in {\mathscr{A}}_\delta$. If $u$ is a smooth uniformly convex solution of (\ref{e1.8}), then there exists a positive constant $C$ depending only on $\Omega$, $\tilde{\Omega}$, $\Lambda_1$, $\Lambda_2$ and $\delta$, such that
\begin{equation}\label{eq3.14}
\max_{\bar{\Omega}}|D^2u| \leq C.
\end{equation}
\end{lemma}

In the following,  we describe the positive lower bound of $D^{2}u$.
For (\ref{eqq2.9}), in consider of the Legendre transformation of $u$, define
$$\tilde L:= \tilde F^{ij}\partial_{ij}+f_{p_i}\partial_i.$$
Then our goal is to show the upper bound of $D^{2}\tilde{u}$ and the argument is very similar to the one used in the proof of Lemma \ref{lem3.5}
by the concavity of $f$ and the condition that $|D f|$ being sufficiently small.
For the convenience of readers, we give the details.

At the beginning of  the repeating procedure, we have
\begin{lemma}\label{lem3.1-1}
Suppose that $f$ is concave on $\Omega$. If $\tilde{u}$ is a smooth uniformly convex solution of (\ref{eqq2.9}), then there holds
\begin{equation}\label{eq3.1a}
   \sup_{\tilde{\Omega}}|D^2\tilde{u}|\leq \max_{\partial\tilde{\Omega}} |D^2\tilde{u}|.
\end{equation}
\end{lemma}
\begin{proof}
For any unit vector $\tilde\xi$, differentiating the equation in (\ref{eqq2.9}) twice in direction $\tilde\xi$ gives
\begin{equation*}
\tilde{L}\tilde{u}_{\tilde\xi\tilde\xi}+\tilde{F}^{ij,rs}\tilde{u}_{ij\tilde\xi}\tilde{u}_{rs\tilde\xi}+\frac{\partial^{2}f}{\partial p_{i}\partial p_{j}}\tilde{u}_{i\tilde\xi}\tilde{u}_{j\tilde\xi}=0.
\end{equation*}
Then by the concavity of $\tilde{F}$ on $\Gamma^+_n$ and $f$ on $\Omega$, we have
\begin{equation*}
\tilde{L}\tilde{u}_{\tilde\xi\tilde\xi}=-\tilde{F}^{ij,rs}\tilde{u}_{ij\tilde\xi}\tilde{u}_{rs\tilde\xi}-\frac{\partial^{2}f}{\partial p_{i}\partial p_{j}}\tilde{u}_{i\tilde\xi}\tilde{u}_{j\tilde\xi}\geq 0.
\end{equation*}
Then by the maximum principle we obtain
$$\sup_{\tilde{\Omega}}\tilde{u}_{\tilde\xi\tilde\xi}\leq \sup_{\partial\tilde{\Omega}}\tilde{u}_{\tilde\xi\tilde\xi}.$$
This completes the proof of (\ref{eq3.1a}).
\end{proof}

Recall that $\tilde{\beta}=(\tilde{\beta}^{1}, \cdots, \tilde{\beta}^{n})$ with $\tilde{\beta}^{k}:=\tilde{h}_{p_{k}}(D\tilde{u})$ and $\tilde{\nu}=(\tilde{\nu}_{1}, \tilde{\nu}_{2},\cdots,\tilde{\nu}_{n})$ is the unit inward normal vector of $\partial\tilde{\Omega}$. Similar to the discussion of (\ref{eq3.2}), (\ref{eq3.3}) and (\ref{eq3.4}), for any tangential direction $\tilde \varsigma$, we have
\begin{equation}\label{eq3.2a}
   u_{\tilde\beta \tilde\varsigma}=\tilde h_{p_k}(D\tilde u)\tilde u_{k\tilde\varsigma}=0.
\end{equation}
Then the second order derivative of $\tilde u$ on the boundary is also controlled by $u_{\tilde\beta \tilde\varsigma}$, $u_{\tilde\beta \tilde\beta}$ and $u_{\tilde\varsigma\tilde\varsigma}$.

At $\tilde x\in \partial\tilde\Omega$, any unit vector $\tilde\xi$ can be written in terms of a tangential component $\tilde\varsigma(\tilde\xi)$ and a component in the direction $\tilde\beta$ by
$$\tilde\xi=\tilde\varsigma(\tilde\xi)+\frac{\langle \tilde\nu,\tilde\xi\rangle}{\langle\tilde\beta,\tilde\nu\rangle}\tilde\beta,$$
where
$$\tilde\varsigma(\tilde\xi):=\tilde\xi-\langle \tilde\nu,\tilde\xi\rangle \tilde\nu-\frac{\langle \tilde\nu,\tilde\xi\rangle}{\langle\tilde\beta,\tilde\nu\rangle}\tilde\beta^T,$$
and
$$\tilde\beta^T:=\tilde\beta-\langle \tilde\beta,\tilde\nu\rangle \tilde\nu.$$
Therefore,
\begin{equation}\label{eq3.3a}
|\tilde\varsigma(\tilde\xi)|\leq C,
\end{equation}
and
\begin{equation}\label{eq3.4a}
u_{\tilde\xi\tilde\xi}\leq C(u_{\tilde\varsigma\tilde\varsigma}+u_{\tilde\beta\tilde\beta}),
\end{equation}
where $\tilde\varsigma:=\frac{\tilde\varsigma(\tilde\xi)}{|\tilde\varsigma(\tilde\xi)|}$ and $C$ depends only on $\Omega$, $\tilde \Omega$, $\Lambda_1$, $\Lambda_2$, $\delta$ and the constant $C_1$ in (\ref{e2.11}). Then we also only need to estimate $u_{\tilde\beta\tilde\beta}$ and $u_{\tilde\varsigma\tilde\varsigma}$ respectively.

Indeed, as shown by Lemma \ref{lem3.2}, we state
\begin{lemma}\label{lem3.2a}
Let $F$ satisfy the structure conditions (\ref{e1.2.0})-(\ref{e1.4}) and $f\in {\mathscr{A}}_\delta$. If $\tilde{u}$ is a smooth uniformly convex solution of (\ref{eqq2.9}) and $|D f|$ is sufficiently small, then there exists a positive constant $C$ depending only on $\Omega$, $\tilde{\Omega}$, $\Lambda_1$, $\Lambda_2$ and $\delta$, such that
\begin{equation}\label{eq3.5a}
   \max_{\partial\Omega}\tilde{u}_{\tilde{\beta}\tilde{\beta}} \leq C.
\end{equation}
\end{lemma}

\begin{proof}
Let $\tilde{x}_0\in\partial\tilde{\Omega}$ satisfy $\tilde{u}_{\tilde{\beta}\tilde{\beta}}(\tilde{x}_0)=\max_{\partial\Omega}\tilde{u}_{\tilde{\beta}\tilde{\beta}}$.
To estimate the upper bound of $\tilde{u}_{\tilde{\beta}\tilde{\beta}}$,
we consider the barrier function
$$\tilde\Psi:=-\tilde{h}(D\tilde{u})+C_0 h+A|y-\tilde{x}_0|^2.$$
For any $y\in \partial\tilde{\Omega}$, $D\tilde{u}(y)\in \partial\Omega$, then $\tilde{h}(D\tilde{u})=0$. It is clear that $h=0$ on $\partial\tilde{\Omega}$. As the proof of (\ref{eqq2.13a})  in terms of $|D f|$ being sufficiently small , we can find the constants $C_0$ and $A$ such that
\begin{equation}\label{eq3.6a}
\begin{cases}
   \tilde{L}\tilde\Psi\leq 0,\quad y\in\tilde{\Omega}_{r},\\
   \tilde\Psi\geq 0 ,\quad\  y\in\partial\tilde{\Omega}_{r}.
\end{cases}
\end{equation}
By the maximum principle, we get
$$\tilde\Psi(y)\geq 0,\quad\quad y\in \tilde{\Omega}_r.$$
Combining it with $\tilde\Psi(\tilde{x}_0)=0$ we obtain $\tilde\Psi_{\tilde{\beta}}(\tilde{x}_0)\geq 0$, which implies
$$\frac{\partial \tilde{h}}{\partial \tilde{\beta}}(D\tilde{u}(\tilde{x}_0))\leq C_0.$$

On the other hand, we see that at $\tilde{x}_0$,
$$\frac{\partial \tilde{h}}{\partial \tilde{\beta}}=\langle D\tilde{h}(D\tilde{u}),\tilde{\beta}\rangle=\frac{\partial \tilde{h}}{\partial p_k}\tilde{u}_{kl}\tilde{\beta}^l=\tilde{\beta}^k\tilde{u}_{kl}\tilde{\beta}^l=\tilde{u}_{\tilde{\beta}\tilde{\beta}}.$$
Therefore,
$$\tilde{u}_{\tilde{\beta}\tilde{\beta}}=\frac{\partial \tilde{h}}{\partial \tilde{\beta}}\leq C.$$
\end{proof}

Next, we estimate the double tangential derivative of $\tilde{u}$.
\begin{lemma}\label{lem3.3a}
Let $F$ satisfy the structure conditions (\ref{e1.2.0})-(\ref{e1.4}) and $f\in {\mathscr{A}}_\delta$. If $\tilde{u}$ is a smooth uniformly convex solution of (\ref{eqq2.9}) and $|D f|$ is sufficiently small, then there exists a positive constant $C$ depending only on $\Omega$, $\tilde{\Omega}$, $\Lambda_1$, $\Lambda_2$ and $\delta$, such that
\begin{equation}\label{eq3.7a}
   \max_{\partial\tilde{\Omega}}\tilde{u}_{\tilde\varsigma\tilde\varsigma} \leq C.
\end{equation}
\end{lemma}

\begin{proof}
Assume that $ \tilde{u}_{\tilde\varsigma\tilde\varsigma}|_{\partial\tilde{\Omega}}$ attains its maximum at $\tilde{x}_0\in\partial \tilde{\Omega}$. Let
$$\tilde M:=\tilde{u}_{11}(\tilde{x}_0)=\max_{\partial\tilde{\Omega}}\tilde{u}_{\tilde\varsigma\tilde\varsigma},$$
and $e_n$ be the unit inward normal vector of $\partial\tilde{\Omega}$ at $\tilde{x}_0$.

For any $y\in \partial\tilde{\Omega}$,  we have by (\ref{eq3.3a}),
\begin{equation}\label{eq3.8a}
\begin{aligned}
\tilde u_{\tilde\xi\tilde\xi}&=|\tilde\varsigma(\tilde\xi)|^2\tilde u_{\tilde\varsigma\tilde\varsigma}+ \frac{\langle \tilde\nu,\tilde\xi\rangle^2}{\langle \tilde\beta,\tilde\nu\rangle^2}\tilde u_{\tilde\beta\tilde\beta}\\
          &\leq \left(1+C\langle \tilde\nu,\tilde\xi\rangle^2-2\langle \tilde\nu,\tilde\xi\rangle \frac{\langle \tilde\beta^T,\tilde\xi\rangle}{\langle \tilde\beta,\tilde\nu\rangle}\right)\tilde M
                 + \frac{\langle \tilde\nu,\tilde\xi\rangle^2}{\langle \tilde\beta,\tilde\nu\rangle^2}\tilde u_{\tilde\beta\tilde\beta}.
\end{aligned}
\end{equation}
Without loss of generality, we assume that $\tilde M\geq 1$. Then by (\ref{e2.11}) and (\ref{eq3.5a}) we have
\begin{equation}\label{eq3.9a}
  \frac{\tilde u_{\tilde\xi\tilde\xi}}{\tilde M}+2\langle \tilde\nu,\tilde\xi\rangle \frac{\langle \tilde\beta^T,\tilde\xi\rangle}{\langle \tilde\beta,\tilde\nu\rangle}
             \leq 1+C\langle \tilde\nu,\tilde\xi\rangle^2.
\end{equation}
Let $\tilde\xi=e_1$, then
\begin{equation}\label{eq3.10a}
  \frac{\tilde u_{11}}{\tilde M}+2\langle \tilde\nu,e_1\rangle \frac{\langle \tilde\beta^T,e_1\rangle}{\langle \tilde\beta,\tilde\nu\rangle}
             \leq 1+C\langle \tilde\nu,e_1\rangle^2.
\end{equation}
We see that the function
\begin{equation}\label{eq3.11a}
\tilde w:=A|y-\tilde x_0|^2-\frac{\tilde u_{11}}{\tilde M}-2\langle \tilde\nu,e_1\rangle \frac{\langle \tilde\beta^T,e_1\rangle}{\langle \tilde\beta,\tilde\nu\rangle}+C\langle \tilde\nu,e_1\rangle^2+1
\end{equation}
satisfies
$$\tilde w|_{\partial\tilde\Omega}\geq 0,\quad  \tilde w(\tilde x_0)=0.$$
Then, by (\ref{eq3.1a}) we can choose the constant $A$ large enough such that
$$\tilde w|_{\tilde\Omega \cap \partial B_{r}(\tilde x_0)} \geq 0.$$

By (\ref{eqq2.9}), $f\in {\mathscr{A}}_\delta$ and $|D f|$ is sufficiently small, we can show that
$$\tilde L\tilde w\leq  C\sum^{n}_{i=1} \tilde F^{ii}.$$
Since the proof of Lemma \ref{lem3.2a}, let us define
$$\tilde\Upsilon:=\tilde w+C_0 {h}.$$
A standard barrier argument makes conclusion of
$$\tilde\Upsilon_{\tilde\beta}(\tilde x_0)\geq0.$$
Therefore,
\begin{equation}\label{eq3.12a}
 \tilde u_{11\tilde\beta}(\tilde x_0)\leq C \tilde M.
\end{equation}
On the other hand, differentiating $\tilde h(D\tilde u)$ twice in the direction $e_1$ at $\tilde x_0$, we have
$$\tilde h_{p_k}\tilde u_{k11}+\tilde h_{p_kp_l}\tilde u_{k1}\tilde u_{l1}=0.$$
The concavity of $\tilde h$ yields that
$$\tilde h_{p_k}\tilde u_{k11}=-\tilde h_{p_kp_l}\tilde u_{k1}\tilde u_{l1}\geq \tilde\theta \tilde M^2.$$
Combining it with $\tilde h_{p_k}\tilde u_{k11}=\tilde u_{11\tilde\beta}$, and using (\ref{eq3.12a}) we obtain
$$\tilde\theta \tilde M^2\leq C\tilde M.$$
Then we get the upper bound of $\tilde M=\tilde u_{11}(\tilde x_0)$ and thus the desired result follows.
\end{proof}

By Lemma \ref{lem3.2a}, Lemma \ref{lem3.3a} and (\ref{eq3.4a}), we obtain the $C^2$ a-priori estimate of $\tilde{u}$ on the boundary.
\begin{lemma}\label{lem3.4a}
Let $F$ satisfy the structure conditions (\ref{e1.2.0})-(\ref{e1.4}) and $f\in {\mathscr{A}}_\delta$. If $\tilde u$ is a smooth uniformly convex solution of (\ref{eqq2.9}) and $|D f|$ is sufficiently small, then there exists a positive constant $C$ depending only on $\Omega$, $\tilde{\Omega}$, $\Lambda_1$, $\Lambda_2$ and $\delta$, such that
\begin{equation}\label{eq3.13a}
\max_{\partial\tilde\Omega}|D^2\tilde u| \leq C.
\end{equation}
\end{lemma}

By Lemma \ref{lem3.1-1} and Lemma \ref{lem3.4a}, one can see that
\begin{lemma}\label{lem3.5a}
Let $F$ satisfy the structure conditions (\ref{e1.2.0})-(\ref{e1.4}) and $f\in {\mathscr{A}}_\delta$. If $\tilde u$ is a smooth uniformly convex solution of (\ref{eqq2.9}) and $|D f|$ is sufficiently small, then there exists a positive constant $C$ depending only on $\Omega$, $\tilde{\Omega}$, $\Lambda_1$, $\Lambda_2$ and $\delta$, such that
\begin{equation}\label{eq3.14a}
\max_{\bar{\tilde\Omega}}|D^2\tilde u| \leq C.
\end{equation}
\end{lemma}

By Lemma \ref{lem3.5} and Lemma \ref{lem3.5a}, we conclude that
\begin{lemma}\label{lem3.6}
Let $F$ satisfy the structure conditions (\ref{e1.2.0})-(\ref{e1.4}) and $f\in {\mathscr{A}}_\delta$. If $u$ is a smooth uniformly convex solution of (\ref{e1.8}) and $|D f|$ is sufficiently small, then there exists a positive constant $C$ depending only on $\Omega$, $\tilde{\Omega}$, $\Lambda_1$, $\Lambda_2$ and $\delta$, such that
\begin{equation}\label{eq3.15}
\frac{1}{C}I_n\leq D^2 u(x) \leq C I_n,\ \ x\in\bar\Omega,
\end{equation}
where $I_n$ is the $n\times n$ identity matrix.
\end{lemma}

\section{Proof of Theorem \ref{t1.2}}

In this section, we will prove Theorem \ref{t1.2}. Two lemmas will be needed when we start the proof. Thanks to the discussion on the uniqueness in the fifth section of \cite{SM}, we deduce that
\begin{lemma}\label{yinli2.4.1}
If $(\ref{e1.2})$ holds and $u\in C^{\infty}(\bar{\Omega})$ be uniformly convex solutions of (\ref{e1.8}), then $u$ is unique up to a constant.
\end{lemma}

Next, similar to Corollary 1.2 in \cite{HRY}, we obtain
\begin{lemma}\label{yinli2.4.2}
If $(\ref{e1.2})$ holds and $u\in C^{2}(\bar{\Omega})$ is a uniformly convex solution of (\ref{e1.8}), then $u\in C^{\infty}(\bar{\Omega})$.
\end{lemma}
\begin{proof}
By Evans-Krylov theorem, it's obvious that
$$u\in C^{\infty}(\Omega).$$
In the following, we describe the regularity of $u$ on the boundary.

Suppose that $x_0\in\partial \Omega$ and in the neighbourhood of $x_0$,
$$\partial \Omega=\{x_n=\varphi(x'): x'\in\Omega'\}$$
for some domain $\Omega'\subset \mathbb{R}^{n-1}$. Without loss of generality, we set $$x_0=(\theta, 0), \varphi(0)=0, D\varphi(0)=\theta, D^2\varphi(0)>0. $$

Let $m_0$ be small enough such that
$$\Omega_{m_0}:=\Omega\cap\{x_n<m_0\}=\{(x',x_n): x'\in \Omega'',\ \varphi(x')<x_n<m_0\}$$
for some domain $\Omega''\subset \mathbb{R}^{n-1}$. And we denote $\Omega''$ also by $\Omega'$.

We take the coordinate transformation in the following
$$T:\Omega_{m_0}\rightarrow T (\Omega_{m_0}),\ x=(x',x_n)\mapsto z=(z',z_n)=(x', x_n-\varphi(x')),$$
where $$T (\Omega_{m_0}):=\{(z',z_n): z'\in\Omega',\ 0<z_n<m_0-\varphi(z')\}.$$
We set $w(z):=u(T^{-1}(z))$.
Then $u(x)=w(x', x_n-\varphi(x'))$. Such that  for $1\leq i\leq n-1$, we see that $u_i=w_i-w_n\varphi_i$ and $u_n=w_n$. Furthermore, for $1\leq k,l\leq n-1$, a direct computation shows that  $$u_{kl}=w_{kl}-w_{kn}\varphi_l-w_{nl}\varphi_k+w_{nn}\varphi_k\varphi_l,$$ $$u_{kn}=w_{kn}-w_{nn}\varphi_k,$$ $$u_{ln}=w_{ln}-w_{nn}\varphi_l,$$ $$ u_{nn}=w_{nn}.$$ It is obvious that by
$$\|D^2 u\|_{C(\bar{\Omega})}+\|D u\|_{C(\bar{\Omega})}\leq C$$
we get
\begin{equation}\label{e2019040801}
\|D^2 w\|_{C(\overline{T(\Omega_{m_0})})}+\|D w\|_{C(\overline{T(\Omega_{m_0})})}\leq C,
\end{equation}
where $C$ is constant depending only on the known data.

Since $w(z)=u(T^{-1}(z))$  and $x=T^{-1}(Z)$, we obtain  $$D^2u=D(DT\cdot Dw),$$ $$ D_{x}u=D_{z}w-w_n(D_{z'}\varphi,0)^t,$$
where $D_{x}, D_{z}$ and $D_{z'}$ denote the gradient according to the variables $x, z$ and $z'$.

Then by (\ref{e1.8}),   $w$ satisfies
\begin{equation}\label{e2019040802}
\left\{ \begin{aligned}
F\left[\lambda(D(DT\cdot Dw))\right]&=f(T^{-1}(z))+c, \ \ &&z\in T(\Omega_{m_0}), \\
h(D_{z}w-w_n(D_{z'}\varphi,0)^t)&=0,\ \ &&z\in \partial T(\Omega_{m_0})\cap \{z_n=0\}.
\end{aligned} \right.
\end{equation}
When $1\leq s\leq n-1$, we set $v=\frac{\partial w}{\partial z_s}$.
Let $a^{ij}(1\leq i,j\leq n), \beta^i(1\leq i \leq n)$ be the definitions in Section 3.

In the following, the repeated indices $k,l$ denote the summation from $1$ to $n-1$  and  the repeated indices $i,j$  present the summation from $1$ to $n$.
Define  \begin{equation*}
[A^{ij}]_{1\leq i,j\leq n}=\left\{ \begin{aligned}
&a^{kl}, &1\leq k,l\leq n-1,\,\,(i,j)=(k,l), \\
&a^{kn}-a^{kl}\varphi_l,  &1\leq k\leq n-1,\,\,(i,j)=(k,n),\\
&a^{nn}+a^{kl}\varphi_k\varphi_l-2a^{kn}\varphi_k,\,\,&(i,j)=(n,n).
\end{aligned} \right.
\end{equation*}
Then differentiating  the equation (\ref{e2019040802}) with respect to $ z_s$,
we see that $v$ solves the following linear equation with oblique boundary condition
\begin{equation}\label{e2019040803}
\left\{ \begin{aligned}
L_s v&=\frac{\partial }{\partial z_s}f(T^{-1}(z)), \ \ &&z\in T(\Omega_{m_0}), \\
\gamma^kv_k+\gamma^nv_n&=\phi,\ \ &&z\in \partial T(\Omega_{m_0})\cap \{z_n=0\},
\end{aligned} \right.
\end{equation}
where
$$L_s=A^{ij}\frac{\partial^{2}}{\partial z_{i}\partial z_{j}},$$
$$\gamma^k=\beta^k (1\leq k\leq n-1), \,\, \gamma^n=\beta^n-\beta^k\varphi_k,$$
$$\phi=w_n\beta^k\varphi_{ks}.$$
By Theorem 1.1 in \cite{LT}, we get $u\in C^{2,\alpha}(\bar\Omega)$ for some $0<\alpha<1$, from which one deduces that
$$A^{ij}\in C^{\alpha}( T(\bar\Omega_{m_0})),$$
$$\gamma^{i}\in C^{1,\alpha}( T(\bar\Omega_{m_0})),$$
$$\phi\in C^{1,\alpha}( T(\bar\Omega_{m_0})).$$
It's obvious that the unit inward normal vector of $\partial\Omega$  at $x_0=(\theta, 0)$ is $\nu=(0,0,\cdots,0,1)$ .
Using lemma \ref{l2.3}, one can shows that $$\beta^n|_{x_0}=\langle\beta, \nu\rangle|_{x_0}\geq \frac{1}{C_{1}}.$$
By $D\varphi(\theta)=0$,  taking $r>0$ small enough such that for any $z'\in B'_r(\theta)\subset\Omega'$, we obtain
$$3n\Lambda_2\geq [A^{ij}]\geq \frac{1}{3n}\Lambda_1.$$
and
$$\langle\gamma, \nu\rangle=\beta^n-\beta^k\varphi_k\geq \frac{1}{3C_{1}},$$
where $\gamma=(\gamma^{1}, \gamma^{2}, \cdots, \gamma^{n})$ and $\nu=(0,0,\cdots,0,1)$ is the unit inward normal vector of
$\partial T(\Omega_{m_0})\bigcap \{z_{n}=0\}$. This implies that the equation (\ref{e2019040803}) is uniformly elliptic with strict oblique boundary condition.

Set $$\mathfrak{G}_{r}:=\{(z',z_n): z'\in B'_r(0),\ 0<z_n<m_0-\varphi(z')\}.$$
Then it follows by Theorem 6.30 and Theorem 6.3.1 in \cite{GT}, for $1\leq s\leq n-1$, we deduce that
$$\frac{\partial w}{\partial z_s}=:v\in C^{2,\alpha}(\overline{\mathfrak{G}_{r}}).$$
Let $Y=\frac{\partial w}{\partial z_n}$.
Then the differentiation of the equation (\ref{e2019040802})  to $ z_n$ gives
$$a^{nn}Y_{nn}+a^{kn}(Y_{kn}-Y_{nn}\varphi_k)+a^{ln}(Y_{ln}-Y_{nn}\varphi_l)+a^{kl}(Y_{kl}-Y_{kn}\varphi_l-Y_{nl}\varphi_k+Y_{nn}\varphi_k\varphi_l)=\frac{\partial f}{\partial z_n}.$$
It reads as
\begin{equation}\label{e2019040804}
(a^{nn}+a^{kl}\varphi_k\varphi_l-a^{kn}\varphi_k-a^{ln}\varphi_l)Y_{nn}=R.H.S,
\end{equation}
where
$$R.H.S:=-a^{kn}Y_{kn}-a^{ln}Y_{ln}-a^{kl}Y_{kl}+a^{kl}Y_{kn}\varphi_l+a^{kl}Y_{nl}\varphi_k+\frac{\partial f}{\partial z_n}.$$
By $v=\frac{\partial w}{\partial z_s}\in C^{2,\alpha}(\overline{\mathfrak{G}_{r}})$, we have $R.H.S\in C^{\alpha}(\overline{\mathfrak{G}_{r}})$.

Recalling the proof of $3n\Lambda_2\geq [A^{ij}]\geq \frac{1}{3n}\Lambda_1$ and taking $r>0$ small enough,  we can also obtain
$$a^{nn}+a^{kl}\varphi_k\varphi_l-a^{kn}\varphi_k-a^{ln}\varphi_l\geq \frac{1}{3n}\Lambda_1.$$
Then by (\ref{e2019040804}) and $(a^{nn}+a^{kl}\varphi_k\varphi_l-a^{kn}\varphi_k-a^{ln}\varphi_l)\in C^{\alpha}(\overline{\mathfrak{G}_{r}})$, we see that $$Y_{nn}\in C^{\alpha}(\overline{\mathfrak{G}_{r}})$$
and therefore $$w\in C^{3,\alpha}(\overline{\mathfrak{G}_{r}}).$$  Consequently, $u\in C^{3,\alpha}(\bar{\Omega}_{m_0})$.

By Finite Covering Theorem, we get $u\in C^{3,\alpha}(\bar\Omega)$. Using bootstrap argument, we obtain $u\in C^{\infty}(\bar\Omega)$.
\end{proof}

Now, by the continuity method, we can show that: \\
\noindent{\bf Proof of Theorem \ref{t1.2}.}
For each $t\in [0,1]$, consider
\begin{equation}\label{e4.1}
\left\{ \begin{aligned}F\left[D^{2}u\right]&=tf(x)+c(t), \ \ x\in \Omega, \\
Du(\Omega)&=\tilde{\Omega},
\end{aligned} \right.
\end{equation}
which is equivalent to
\begin{equation}\label{e4.2}
\left\{ \begin{aligned}F\left[D^{2}u\right]&=tf(x)+c(t), \ \ &&x\in \Omega, \\
h(Du)&=0,\ \ &&x\in\partial \Omega.
\end{aligned} \right.
\end{equation}
By the main results in \cite{HO} and \cite{HRY}, (\ref{e4.1}) is solvable if $t=0$. By Lemma \ref{yinli2.4.2}, denote the closed subset
$$X:=\{u\in C^{2,\alpha}(\bar{\Omega}): \int_{\Omega}u=0\}$$
in $C^{2,\alpha}(\bar{\Omega})$ and
$$Y:= C^{\alpha}(\bar{\Omega})\times C^{1,\alpha}(\partial\Omega).$$
Define a map from $X\times \mathbb{R}$ to $Y$ as
$$\mathfrak{F}^{t}:=\left(F(D^{2}u)-tf(x)-c(t),h(Du)\right).$$
Then the linearized operator $D\mathfrak{F}^{t}_{(u,c)}:X\times \mathbb{R}\rightarrow Y$ is given by
$$D\mathfrak{F}^{t}_{(u,c)}(w,a)=\left(F^{ij}(D^2u)\partial_{ij}w-a, h_{p_i}(Du)\partial_i w\right).$$
Repeating the proof of Proposition 3.1 in \cite{SM}, we know that $D\mathfrak{F}^{t}_{(u,c)}$ is invertible for any $t\in[0,1]$ and $(u,c)$ being the solution to (\ref{e4.1}).

Define the set
$$I:=\left\{t\in[0,1]:\text{(\ref{e4.1}) has at least one convex solution}\right\}.$$
Since $0\in I$, by Huang-Ou's theorem \cite{HO} or Huang-Ye's theorem \cite{HRY}, $I$ is not empty. We claim that $I=[0,1]$, which is equivalent to the fact that $I$ is not only open, but also closed. It follows from Proposition 3.1 in \cite{SM} again and Theorem 17.6 in \cite{GT} that $I$ is open. So we only need to prove that $I$ is a closed subset of $[0,1]$.

That $I$ is closed is equivalent to the fact that for any sequence $\{t_k\}\subset I$, if $\lim_{k\rightarrow \infty}t_k=t_0$, then $t_0\in I$. For $t_k$, denote $(u_k, c(t_k))$ solving
\begin{equation*}
\left\{ \begin{aligned}F\left[D^{2}u_k\right]&=t_kf(x)+c(t_k), \ \ x\in \Omega, \\
Du_k(\Omega)&=\tilde{\Omega}.
\end{aligned} \right.
\end{equation*}
It follows from Lemma \ref{lem3.6} that $\|u_k\|_{C^{2,\alpha}(\bar{\Omega})}\leq C$, where $C$ is independent of $t_k$. Since
$$|c(t)|=\left|F\left[D^{2}u\right]-tf(x)\right|\leq F\left(+\infty,\cdots,+\infty\right)+\max_{\bar{\Omega}}|f(x)|, $$
by Arzela-Ascoli Theorem we know that there exists $\hat{u}\in C^{2,\alpha}(\bar{\Omega})$, $\hat{c}\in\mathbb{R}$ and a subsequence of $\{t_k\}$, which is still denoted as $\{t_k\}$, such that letting $k\rightarrow \infty$,
\begin{equation*}
\left\{ \begin{aligned}
&\left\|u_k-\hat{u}\right\|_{C^{2}(\bar{\Omega})}\rightarrow 0, \\
&c(t_k)\rightarrow \hat{c}.
\end{aligned} \right.
\end{equation*}
Since $(u_k, c(t_k))$ satisfies
\begin{equation*}
\left\{ \begin{aligned}F\left[D^{2}u_k\right]&=t_kf(x)+c(t_k), \ \ &&x\in \Omega, \\
h(Du_k)&=0,\ \ &&x\in\partial \Omega,
\end{aligned} \right.
\end{equation*}
letting $k\rightarrow \infty$, we have
\begin{equation*}
\left\{ \begin{aligned}F\left[D^{2}\hat{u}\right]&=t_0f(x)+\hat{c}, \ \ &&x\in \Omega, \\
h(D\hat{u})&=0,\ \ &&x\in\partial \Omega.
\end{aligned} \right.
\end{equation*}
Therefore, $t_0\in I$, and thus $I$ is closed. Consequently, $I=[0,1]$. By Lemma \ref{yinli2.4.1} we know that the solution of (\ref{e4.1})  $u$ is unique up to a constant.

Then we complete the proof of Theorem \ref{t1.2}.
\qed

\section*{Acknowledgements}
The authors would like to express deep gratitude to Professor Yuanlong Xin for his suggestions and constant encouragement.

\end{document}